\documentclass[12pt,reqno]{amsart}

\usepackage{a4wide}
\usepackage{amsmath,amssymb,amsfonts,amsthm}
\usepackage{enumerate}
\usepackage{mathtools}
\usepackage{graphicx}
\usepackage{caption}
\usepackage{url}

\usepackage[
  pdfauthor={Bernd C. Kellner},
  pdfkeywords={Integer-valued functions, iterations, matrix, determinant, 3x+1 problem},
  pdftitle={Sparse matrices describing iterations of integer-valued functions},
  linktocpage,colorlinks,bookmarksnumbered]{hyperref}

\newcommand{\ZZ}{\mathbb{Z}}                
\newcommand{\NN}{\mathbb{N}}                
\newcommand{\PP}{\mathbb{P}}                
\newcommand{\RD}{\mathbb{D}}                
\newcommand{\RDC}{\widetilde{\mathbb{D}}}   
\newcommand{\Cyc}{\mathcal{C}}              
\newcommand{\Vect}[1]{\mathbf{#1}}          
\newcommand{\MA}{M}                         
\newcommand{\MI}{\widehat{M}}               
\newcommand{\Set}[1]{\{#1\}}                

\DeclareMathOperator*{\BigOCup}{\mathring{\bigcup}}

\newtheorem{Theorem}{Theorem}[section]
\newtheorem{Proposition}[Theorem]{Proposition}
\newtheorem{Lemma}[Theorem]{Lemma}
\newtheorem{Corollary}[Theorem]{Corollary}

\theoremstyle{remark}
\newtheorem{Remark}[Theorem]{Remark}
\newtheorem*{Remark*}{Remark}

\numberwithin{equation}{section}
\numberwithin{figure}{section}

\title{Sparse matrices describing iterations of integer-valued functions}
\author{Bernd C. Kellner}
\email{bk@bernoulli.org}


\begin{document}

\begin{abstract}
We consider iterations of integer-valued functions $\phi$, which have no fixed
points in the domain of positive integers. We define a local function $\phi_n$,
which is a sub-function of $\phi$ being restricted to the subdomain
$\{ 0, \ldots, n \}$. The iterations of $\phi_n$ can be described by a certain
$n \times n$ sparse matrix $M_n$ and its powers. The determinant of the related
$n \times n$ matrix $\widehat{M}_n = I - M_n$, where $I$ is the identity matrix,
acts as an indicator, whether the iterations of the local function $\phi_n$
enter a cycle or not. If $\phi_n$ has no cycle, then $\det \widehat{M}_n = 1$
and the structure of the inverse $\widehat{M}_n^{-1}$ can be characterized.
Subsequently, we give applications to compute the inverse $\widehat{M}_n^{-1}$
for some special functions. At the end, we discuss the results in connection
with the $3x+1$ and related problems.
\end{abstract}

\maketitle


\section{Introduction}

Let $\ZZ$ and $\NN$ be the set of integers and positive integers, respectively.
Let $\NN_0 = \NN \cup \Set{ 0 }$. Let $n, m$ denote positive integers in this
paper. Define the finite domain $\RD_n = \Set{ 1, \ldots, n }$ and let
$\RD_{n,0} = \RD_n \cup \Set{ 0 }$.

We consider integer-valued functions on domains $S \subset \NN$, where
\[
  \phi_S : S \to \NN, \quad \phi(x) \neq x \quad (x \in S).
\]
We may define a function $\phi$ induced by $\phi_S$ by
\[
  \phi(x) = \begin{cases}
    \phi_S(x), & \text{if } x \in S, \\
    0,       & \text{else},
  \end{cases}
\]
which has the properties that
\begin{equation} \label{eq:phi-prop}
  \phi : \NN_0 \to \NN_0, \quad \phi(0) = 0,
    \quad \phi(x) \neq x \quad (x \in \NN),
\end{equation}
having no fixed points in $\NN$. Let $\Phi$ be the set of all
such functions satisfying \eqref{eq:phi-prop}.

Note that a function $\phi$ does not have to be analytic, i.e.\ being an
integer-valued polynomial. It can be arbitrarily defined, for example, using
piecewise functions or tables of any complexity.

We construct a local function $\phi_n$, that is
a sub-function of $\phi \in \Phi$ being restricted to the domain $\RD_{n,0}$, by
\[
  \phi_n : \RD_{n,0} \to \RD_{n,0}, \quad
  \phi_n(x) = \begin{cases}
    \phi(x), & \text{if } ( x, \phi(x) ) \in \RD^2_n, \\
    0,       & \text{else}.
  \end{cases}
\]

We denote $\phi^m = \phi \circ \cdots \circ \phi$ as an $m$-fold iteration of
$\phi$. We say that an iteration stops, if there exists an index $k$ such that
$\phi^k(x) = 0$, since all successive values of the iteration also vanish
by definition. We say that $\phi$ has a cycle of length $m \geq 2$,
if there exists $x \in \NN$ such that $\phi^m(x) = x$ and $\phi^{m'}(x) \neq x$
for $m' < m$. If $\phi$ has a cycle of length $m$ containing $x$, then we may
define the set
\begin{equation} \label{eq:cycle-def}
  \Cyc(\phi,m,x) = \Set{ \phi(x), \ldots, \phi^m(x) },
\end{equation}
describing all elements of this cycle. By definition this set has the properties
that
\begin{equation} \label{eq:cycle-prop}
  x \in \Cyc(\phi,m,x) \quad \text{and} \quad | \Cyc(\phi,m,x) | = m \geq 2.
\end{equation}
Regarding $\phi_n$, we require that $x \in \RD_n$ to define a cycle
$\Cyc(\phi_n,m,x)$ of length $m \geq 2$. Note that $\phi_n$ cannot have
a fixed point $x \in \RD_n$ by definition.

\begin{Lemma} \label{lem:cycle}
Let $\phi \in \Phi$. Assume that there exists a cycle $\Cyc(\phi,m,x)$. Then
\[
  \Cyc(\phi,m,x) = \Cyc(\phi_n,m,x)
    \quad \Longleftrightarrow \quad
    n \geq \max \, \Cyc(\phi,m,x).
\]
\end{Lemma}

\begin{proof}
Set $N = \max \, \Cyc(\phi,m,x)$. If $n \geq N$, then we have
\[
  (\phi(x),\phi^2(x)), \ldots, (\phi^{m-1}(x),\phi^m(x)) \in \RD^2_n,
\]
which also holds for $\phi_n$ having the same function values by definition.
Therefore $\phi_n$ has the same  cycle as $\phi$ in this case.

Conversely, if $n < N$, then there exists an index $i$, such that $\phi^i(x) > n$
showing that $\phi^i(x), \phi^i_n(x) \notin \RD_n$, where $\phi^i_n(x) = 0$.
Accordingly, $\Cyc(\phi,m,x) = \Cyc(\phi_n,m,x)$ implies that $n \geq N$ must
hold.
\end{proof}

\begin{Lemma} \label{lem:height}
Let $\phi \in \Phi$ and $n \geq 1$. Assume that $\phi_n$ has no cycle.
We define the height of $x \in \RD_n$ regarding $\phi_n$ by
\[
  h(x) = \min \, \Set{ k \in \NN : \phi_n^k( x ) = 0 }.
\]
We then have
\[
  1 \leq h(x) \leq n \quad (x \in \RD_n).
\]
\end{Lemma}

\begin{proof}
Let $x \in \RD_n$. The property $h(x) \geq 1$ follows by definition. Next, we
show that $h(x) \leq n$. Assume to the contrary that $\phi_n^n(x) \in \RD_n$
implying that $h(x) > n$. Since $\phi_n$ has no cycle, we then obtain that
\[
  \mathcal{X} = \Set{ x, \phi_n(x), \ldots, \phi^n_n(x) } \subset \RD_n,
\]
where $\mathcal{X}$ must contain $n+1$ distinct elements. This gives a
contradiction to $|\RD_n| = n$.
\end{proof}

We define the $i$-th unit vector of size $n$ by $\Vect{e}_i$ and the zero
vector by $\Vect{e}_0$, which we shall use in an unambiguous context. Let $A, B,
E_n(\cdot,\cdot), I_n$ be $n \times n$ matrices. As usual, $I_n$ denotes the
identity matrix, where we use $I$ instead, if possible. Define the matrix
\[
  E_n(i,j) = \Vect{e}_i \Vect{e}_j^t,
\]
which only has the entry $1$ at row $i$ and column $j$ and zeros elsewhere. Let
$\# A$ denote the total number of nonzero entries of $A$. The term $\# A^k$
should be read as $\#(A^k)$. We define for $A$ and $B$ that
\[
  A \cap B = \sum_{\substack{1 \leq i,j \leq n\\A_{ij} B_{ij} \neq 0}} 1,
\]
counting all entries, where both matrices have nonzero entries in common.
We call $A$ and $B$ to be disjoint, if
\[
  A \cap B = 0
\]
implying that
\[
  \#( A + B ) = \# A + \# B.
\]

The main aim of the paper is to construct $n \times n$ matrices, which are
connected with the properties of a function $\phi \in \Phi$ as well as its local
function $\phi_n$. We define the following matrix by column vectors induced by
the local function $\phi_n$ by
\begin{equation} \label{eq:matrix-mn-def}
  \MA_n( \phi ) = \Bigg( \Vect{e}_{\phi_n(1)}, \ldots, \Vect{e}_{\phi_n(n)}
    \Bigg) \in \ZZ^{n \times n}
    \quad (\phi \in \Phi, n \geq 2),
\end{equation}
being a binary matrix with $\Set{ 0, 1 }$ entries. Further we define the related
matrix
\[
  \MI_n( \phi ) = I - M_n( \phi ) \in \ZZ^{n \times n}
    \quad (\phi \in \Phi, n \geq 2),
\]
which consists of entries with $\Set{ -1, 0, 1 }$. Since $\phi_n$ has no fixed
points in $\RD_n$, the diagonal of $\MA_n( \phi )$ and $\MI_n( \phi )$ has only
entries with $0$ and $1$, respectively. By construction the matrices
$\MA_n( \phi )$ and $\MI_n( \phi )$ are sparse matrices, since
$\# \MA_n( \phi ) \leq n$ and $\# \MI_n( \phi ) \leq 2n$, both being of order
$O(n)$.

The main property of $\MA_n( \phi )$ is that for $x \in \RD_n$ the mapping
\begin{align}
  x &\mapsto \phi_n(x) \nonumber \\
\intertext{coincides with}
  \MA_n( \phi ) \, \Vect{e}_x &= \Vect{e}_{\phi_n(x)}.
    \label{eq:matrix-vect}
\end{align}
If there exists a cycle of $\phi$, then $\det \MI_n( \phi )$ acts as an indicator
for this event. This is shown by the following theorems.

\begin{Theorem} \label{thm:matrix-cycle}
Let $\phi \in \Phi$ and $n \geq 2$.
If there exists a cycle $\Cyc(\phi_n,m,x)$, then
\[
  \det \MI_n( \phi ) = 0
\]
and $\MA_n( \phi )$ has an eigenvector $\Vect{v}$ with eigenvalue $1$ defined by
\[
  \Vect{v} = \sum_{y \in \Cyc(\phi_n,m,x)} \Vect{e}_y.
\]
\end{Theorem}

\begin{Corollary} \label{corl:matrix-cycle}
Let $\phi \in \Phi$. If $\phi$ has a cycle, then there exists an integer
$N \geq 2$ such that
\[
  \det \MI_n( \phi ) = 0 \quad (n \geq N).
\]
\end{Corollary}

\begin{proof}
By assumption $\phi$ has a cycle, say $\Cyc(\phi,m,x)$.
By Lemma~\ref{lem:cycle} we can find an integer $N \geq 2$ such that
\[
  \Cyc(\phi,m,x) = \Cyc(\phi_n,m,x) \quad (n \geq N).
\]
Applying Theorem~\ref{thm:matrix-cycle} for $n \geq N$ gives the result.
\end{proof}

\begin{Theorem} \label{thm:matrix-inverse}
Let $\phi \in \Phi$ and $n \geq 2$. If $\phi_n$ has no cycle, then we have the
following statements:
\begin{enumerate}
\item The matrix $\MA_n( \phi )$ is nilpotent of degree at most $n$.
\item The powers of $\MA_n( \phi )$ are binary matrices satisfying
\[
  \# \MA_n( \phi )^k \leq n-k \quad (1 \leq k \leq n).
\]
They are disjoint for different exponents that
\[
  \MA_n( \phi )^k \cap \MA_n( \phi )^l = 0 \quad (k \neq l, \, k,l \geq 1).
\]
\item The related matrix $\MI_n( \phi )$ is invertible, where
\[
  \det \MI_n( \phi ) = 1.
\]
\item The inverse $\MI_n( \phi )^{-1}$ is a binary matrix with the properties
that
\[
  \MI_n( \phi )^{-1} = I + \MA_n( \phi ) + \cdots + \MA_n( \phi )^{n-1}
\]
and
\[
  \# \MI_n( \phi )^{-1} = n + \sum_{k=1}^{n-1} \# \MA_n( \phi )^k
    \leq \binom{n+1}{2}.
\]
\end{enumerate}
\end{Theorem}

\begin{Remark}
The bounds of Theorem~\ref{thm:matrix-inverse} are sharp. If we consider the
function $\phi \in \Phi$ being induced by $f(x) = x+1$, then $\MA_n( \phi )$ is
nilpotent of degree $n$, $\# \MA_n( \phi )^k = n-k$ for $1 \leq k \leq n$, and
consequently $\# \MI_n( \phi )^{-1} = \binom{n+1}{2}$. This will be shown by
Proposition~\ref{prop:phi-x+1} as an example. The next theorem shows that one
can compute exact values for an arbitrary local function $\phi_n$, in case it
has no cycle.
\end{Remark}

\begin{Theorem} \label{thm:matrix-part}
Let $\phi \in \Phi$ and $n \geq 2$. Assume that $\phi_n$ has no cycle. Let
\[
  p_\nu = | \Set{ x \in \RD_n : h(x) = \nu } |
    \quad (1 \leq \nu \leq n)
\]
and
\[
  m = \max_{1 \leq \nu \leq n} \, \Set{ \nu : p_\nu \geq 1 }.
\]
Define
\[
  \pi = (p_1, \ldots, p_m).
\]
Then $\pi$ is an ordered partition of $n$ and length $m$ with $1 \leq m \leq n$,
where $p_1, \ldots, p_m \geq 1$. We have the following statements:
\begin{enumerate}
\item The matrix $\MA_n( \phi )$ is nilpotent of degree $m$.
\item We have
\[
  \MA_n( \phi )^k = \sum_{( \phi^k_n(x), x ) \in \RD_n^2} E_n( \phi^k_n(x), x )
    \quad(1 \leq k \leq m),
\]
where
\[
  \# \MA_n( \phi )^k = n - \sum_{\nu = 1}^k p_\nu
    \quad (1 \leq k \leq m).
\]
\item We have
\[
  \MI_n( \phi )^{-1} = \Bigg( \Vect{v}_1, \ldots, \Vect{v}_n \Bigg),
\]
where the column vectors satisfy that
\[
  \Vect{v}_j = \Vect{e}_j + \sum_{\nu = 1}^{h(j)-1} \Vect{e}_{\phi_n^\nu(j)}
    \quad (1 \leq j \leq n).
\]
Moreover,
\[
  \# \MI_n( \phi )^{-1}
    = \sum_{x \in \RD_n} h(x)
    = \sum_{\nu=1}^m \nu p_\nu \leq n m - \binom{m}{2}.
\]
\end{enumerate}
\end{Theorem}

An ordered partition of an integer, where the order of its summands is relevant,
is also called a composition (cf.~\cite[Chap.~II, p.~123]{Comtet:1974}).
We will prove the theorems above in the following sections. We will give some
applications of the theorems in Sections~\ref{sec:simple-patterns} and
\ref{sec:3x+1-problem}. See the figures therein for illustrations of examples of
the matrices $\MI_n( \phi )$ and $\MI_n( \phi )^{-1}$.


\section{Iterations}

Recall that $\phi_n: \RD_{n,0} \to \RD_{n,0}$ can be any \textsl{exotic}
function, i.e.\ $\phi_n$ can be composed of piecewise functions of any
complexity. This situation will be reflected in the following lemmas and
propositions in this section. First, we define an orbit of an element
$x \in \RD_n$, collecting the iterations $\phi_n^k(x)$. Regarding these orbits
we can define an equivalence relation on $\RD_n$, which leads to a disjoint
decomposition of $\RD_n$. Second, these disjoint sets in question can be
interpreted as labeled trees, whose properties establish the results. For basic
graph theory see \cite[Chap.~I.17, pp.~60]{Comtet:1974}.

\begin{Lemma} \label{lem:phi-k-neq-l}
Let $\phi \in \Phi$ and $k,l,n \in \NN$, where $k \neq l$.
If $\phi$ has no cycle, then
\[
  \phi^k(x) \neq \phi^l(x) \quad (x, \phi^k(x), \phi^l(x) \in \NN).
\]
Accordingly, if $\phi_n$ has no cycle, then
\[
  \phi_n^k(x) \neq \phi_n^l(x) \quad (x, \phi_n^k(x), \phi_n^l(x) \in \RD_n).
\]
\end{Lemma}

\begin{proof}
By symmetry we may assume that $k > l$. Set $y = \phi^l(x)$. Assume to the
contrary that we have $\phi^k(x) = \phi^l(x)$. This implies that $\phi^{k-l}(y)
= y$, contradicting that $\phi$ has no cycle or no fixed point. Similarly, if
$x, \phi_n^k(x), \phi_n^l(x) \in \RD_n$, then $\phi_n^k(x) = \phi_n^l(x)$ also
would imply, that $\phi_n$ has a cycle or a fixed point.
\end{proof}

\begin{Lemma} \label{lem:orbit-decomp}
Let $\phi \in \Phi$ and $n \geq 1$. Assume that $\phi_n$ has no cycle.
Define the orbit of $x \in \RD_n$ by
\[
  \Omega(x) = \Set{ \phi_n^k( x ) : 0 \leq k < h(x) } \subset \RD_n,
\]
where $\phi_n^0( x ) = x$ is defined to be the identity function on $\RD_n$ and
$h(x)$ is the height of $x$. Then
\[
  1 \leq | \Omega(x) | = h(x) \leq n \quad (x \in \RD_n).
\]
There exists an equivalence relation on $\RD_n$ induced by $\phi_n$, such that
\begin{equation} \label{eq:equiv-x-omega}
  x \sim y \quad \Longleftrightarrow \quad
    \Omega(x) \cap \Omega(y) \neq \emptyset
    \quad (x, y \in \RD_n).
\end{equation}
Let denote $[x]$ the equivalence class of $x$ and
\[
  \RDC_n = \RD_n /\!\! \sim
\]
the set of these classes. Define the set
\begin{equation} \label{eq:cover-orbit}
  \Omega_{[x]} = \bigcup_{x' \in [x]} \Omega( x' ),
\end{equation}
which covers all orbits of $x' \in [x]$. Then
\begin{equation} \label{eq:disjoint-decomp}
  \RD_n = \BigOCup_{[x] \in \RDC_n} \Omega_{[x]}
\end{equation}
gives a disjoint decomposition.
\end{Lemma}

\begin{proof}
By Lemma~\ref{lem:height} we have the bounds
\[
  1 \leq h(x) \leq n \quad (x \in \RD_n).
\]
Lemma~\ref{lem:phi-k-neq-l} implies that
\[
  1 \leq | \Omega(x) | = h(x) \leq n \quad (x \in \RD_n).
\]
It is easy to see that
\begin{equation} \label{eq:equiv-x-height}
  x \sim y \quad \Longleftrightarrow \quad
   \phi_n^{h(x)-1}( x ) = \phi_n^{h(y)-1}( y )
   \quad (x, y \in \RD_n)
\end{equation}
defines an equivalence relation on $\RD_n$. We show that this relation transfers
to \eqref{eq:equiv-x-omega} as follows. If there exists
\[
  z \in \Omega(x) \cap \Omega(y) \neq \emptyset,
\]
then we have
\[
  \Omega(z) = \Set{ z, \phi_n(z), \ldots, \phi_n^{h(z)-1}( z )}
    \subset \Omega(x) \cap \Omega(y).
\]
By construction of $\Omega(\cdot)$ and Lemma~\ref{lem:phi-k-neq-l}, the
elements of an orbit can be uniquely ordered by height. Thus, there exists only
one element of height $1$ in $\Omega(\cdot)$. We then infer that
\[
  \phi_n^{h(z)-1}( z ) = \phi_n^{h(x)-1}( x ) = \phi_n^{h(y)-1}( y ),
\]
showing that $x \sim y$. Conversely, if $x \sim y$, then
\[
  \phi_n^{h(x)-1}( x ), \phi_n^{h(y)-1}( y ) \in \Omega(x) \cap \Omega(y)
    \neq \emptyset,
\]
using the same arguments as above. After all, this shows \eqref{eq:equiv-x-omega}.
The construction of $\Omega_{[x]}$ in \eqref{eq:cover-orbit} implies that
\[
  \Omega_{[x]} \cap \Omega_{[y]} = \emptyset,
    \quad \text{if\ } [x] \neq [y],
\]
applying \eqref{eq:equiv-x-omega}. This finally establishes the disjoint
decomposition in \eqref{eq:disjoint-decomp}.
\end{proof}

\begin{Lemma} \label{lem:tree}
Let $\phi \in \Phi$ and $n \geq 1$. Assume that $\phi_n$ has no cycle.
Define
\[
  \hat{\phi}_n( (x, k) ) = ( \phi_n(x), k-1 ) \quad (x \in \RD_n, k \geq 1).
\]
Let $[x] \in \RDC_n$. Then there exists a labeled tree
\[
  T_{[x]} \cong  \Omega_{[x]},
\]
which can be uniquely defined by its labeled nodes by
\begin{equation} \label{eq:tree-def}
  T_{[x]} = \Set{ (y,h(y)) :  y \in \Omega_{[x]} },
\end{equation}
preserving the structure induced by $\hat{\phi}_n$. The tree $T_{[x]}$ has the
following properties:
\begin{enumerate}
\item There exists a unique root node $(e,1) \in T_{[x]}$,
      where $\hat{\phi}_n( (e,1) ) = (0, 0)$.
\item If a node $(d,k) \in T_{[x]}$ $(k \geq 1)$ has $m$ child nodes, then
      they are given by \newline
      $\Set{(c_i,k+1)}_{1 \leq i \leq m}$ satisfying
      $\hat{\phi}_n( (c_i,k+1) ) = (d,k) \quad (1 \leq i \leq m)$.
\end{enumerate}
\end{Lemma}

\begin{proof}
By construction of $\Omega_{[x]}$, we obviously have a one-to-one correspondence
\begin{equation} \label{eq:y-hy}
  y \longleftrightarrow (y,h(y)) \quad (y \in \Omega_{[x]}).
\end{equation}
We can define a labeled tree by
\[
  T_{[x]} = \Set{ (y,h(y)) :  y \in \Omega_{[x]} },
\]
labeling the nodes with $(y,h(y))$ uniquely. We will show that
\[
  T_{[x]} \cong  \Omega_{[x]}
\]
having the same structure. By Lemmas \ref{lem:phi-k-neq-l} and \ref{lem:orbit-decomp}
we infer that an orbit $\Omega(y)$ represents a simple path
\begin{equation} \label{eq:simple-path}
  \Set{ y, \phi_n(y), \ldots, \phi_n^{h(y)-1}( y ) },
\end{equation}
where the elements are ordered by height. The set $\Omega_{[x]}$ covers all
orbits $\Omega(y)$, respectively simple paths as in \eqref{eq:simple-path},
where $y \in [x]$. Since the equivalence relation on $\RD_n$ satisfies
\eqref{eq:equiv-x-height}, all such simple paths are closed by a common
element $e$, where
\begin{equation} \label{eq:elem-e}
  e = \phi_n^{h(y)-1}( y ) \quad \text{for all } y \in [x],
\end{equation}
being the only element with height $h(e) = 1$ in $\Omega_{[x]}$. The structure
of $\Omega_{[x]}$ is induced by $\phi_n$, where each element
$y \in \Omega_{[x]}$ has its height and a successor $\phi_n(y) \in \Omega_{[x]}$,
if $y \neq e$. By \eqref{eq:y-hy} and the definition of $\hat{\phi}_n$ the
properties of $\Omega_{[x]}$ are transferred to $T_{[x]}$. As a result, the
tree $T_{[x]}$ is a cover of simple paths. Each node $(y, h(y)) \in T_{[x]}$
has a parent node $\hat{\phi}_n( (y, h(y)) ) = ( \phi_n(y), h(y)-1 ) \in T_{[x]}$,
if $y \neq e$.

Property (1): By \eqref{eq:elem-e} there exists only one element $e$ with
$h(e) = 1$. Thus, $(e,1) \in T_{[x]}$ is the unique root node with
$\hat{\phi}_n( (e,1) ) = (0, 0)$.

Property (2): We assume that $(d,k) \in T_{[x]}$ has $m$ child nodes. Since
$T_{[x]}$ is a cover of simple paths, there exist $m$ different paths, such that
\[
  \Set{(c_i, h(c_i)), \hat{\phi}_n( (c_i, h(c_i)) ) = (d, h(c_i)-1),
    \ldots, (e, 1)} \subset T_{[x]}
    \quad (1 \leq i \leq m),
\]
where $h(c_1) = \cdots = h(c_m) = k + 1$.
\end{proof}

\begin{Proposition} \label{prop:jnk-sets}
Let $\phi \in \Phi$ and $k, n \geq 1$. Assume that $\phi_n$ has no cycle.
Define
\[
  J_{n,k}( \phi ) = \Set{ x \in \RD_n : \phi_n^k( x ) \in \RD_n }.
\]
Then we have
\[
  | J_{n,k}( \phi ) | \leq \max (n-k, 0).
\]
In particular,
\[
  | J_{n,1}( \phi ) | = n - | \RDC_n |.
\]
\end{Proposition}

\begin{proof}
Since $\phi_n$ has no cycle, we have by Lemma~\ref{lem:height} that
\[
  \phi_n^n( x ) = 0 \quad (x \in \RD_n),
\]
which implies that
\begin{equation} \label{eq:jnk-0}
  | J_{n,k}( \phi ) | = 0 \quad (k \geq n).
\end{equation}
Define the complementary set
\[
  \overline{J_{n,k}}( \phi ) = \Set{ x \in \RD_n : \phi_n^k( x ) = 0 },
\]
such that
\begin{equation} \label{eq:jnk-compl-union}
  J_{n,k}( \phi ) \cup \overline{J_{n,k}}( \phi ) = \RD_n.
\end{equation}
In view of \eqref{eq:jnk-0} and \eqref{eq:jnk-compl-union}, we will equivalently
show for the remaining cases that
\begin{equation} \label{eq:jnk-compl-bound}
  | \overline{J_{n,k}}( \phi ) | \geq k \quad (1 \leq k < n).
\end{equation}
By Lemmas \ref{lem:orbit-decomp} and \ref{lem:tree}, there exist unique labeled
trees $T_{[x]}$ for $[x] \in \RDC_n$, where
\[
  \sum_{[x] \in \RDC_n} | T_{[x]} | = n.
\]
More precisely, they build a disjoint decomposition of $\RD_n$, if we only
consider the first component of the elements $(y,h(y)) \in T_{[x]}$.
We will count those elements in all trees, that vanish under $\phi_n^k$. Since
the elements can be ordered by height, we can write
\begin{equation} \label{eq:jnk-compl-count}
  | \overline{J_{n,k}}( \phi ) |
    = | \Set{ x \in \RD_n : \phi_n^k( x ) = 0 } |
    = \sum_{[x] \in \widetilde{\RD}_n}
      \sum_{\substack{(y,j) \in T_{[x]}\\j \leq k}} 1.
\end{equation}
In particular, we obtain by counting the root nodes via
\eqref{eq:jnk-compl-count} and using \eqref{eq:jnk-compl-union} that
\begin{equation} \label{eq:jnk-compl-1}
  | \overline{J_{n,1}}( \phi ) | = | \widetilde{\RD}_n | \geq 1
    \quad \text{and} \quad
    | J_{n,1}( \phi ) | = n - | \RDC_n |,
\end{equation}
showing the special case $k=1$. Define $\overline{J_{n,0}}( \phi ) = \emptyset$.
Instead of \eqref{eq:jnk-compl-count}, it is more convenient to consider
\begin{equation} \label{eq:jnk-compl-sk}
  s_k = | \overline{J_{n,k}}( \phi ) \backslash \overline{J_{n,k-1}}( \phi ) |
      = \sum_{[x] \in \widetilde{\RD}_n}
        \sum_{(y,k) \in T_{[x]}} 1
    \quad (1 \leq k \leq n),
\end{equation}
where $s_k$ is the number of elements having height $k$. We derive that
\[
  \sum_{k=1}^n s_k = | \overline{J_{n,n}}( \phi ) | = n,
\]
utilizing the telescoping sum induced by \eqref{eq:jnk-compl-sk}, where the
last equation follows by \eqref{eq:jnk-0} and \eqref{eq:jnk-compl-union}.
Hence, $S = (s_1, \ldots, s_n)$ describes an ordered partition of $n$, where
$s_1 \geq 1$ follows by \eqref{eq:jnk-compl-1}. Now, we use the properties of a
tree $T_{[x]}$ given by Lemma~\ref{lem:tree}:
\begin{enumerate}
\item A tree $T_{[x]}$, being a labeled tree, is connected.
\item The heights of a parent node $(d, h(d)) \in T_{[x]}$ and its child nodes
      $(c_i,h(c_i)) \in T_{[x]}$ differ by $1$, such that $h(c_i) = h(d)+1$.
\end{enumerate}
As a consequence, we infer that the sequence $(s_\nu)_{1 \leq \nu \leq n}$,
counting elements of height $1$ up to $n$, cannot have gaps, i.e.\ zero
elements, in the middle, such that
\[
  S = (s_1, \ldots, s_{l-1}, 0, s_{l+1}, \ldots, s_n)
\]
with some index $l$, where $1 < l < n$, and $s_1, s_n \geq 1$.
Therefore, we either must have
\begin{equation} \label{eq:jnk-s-part}
  S = (1, \ldots, 1)
    \quad \text{or} \quad
    S = (s_1, \ldots, s_r, 0, \ldots, 0),
\end{equation}
where $1 \leq r < n$ and $s_1, \ldots, s_r \geq 1$.
In any case, we finally conclude by an easy counting argument that
\begin{equation} \label{eq:jnk-sum}
  | \overline{J_{n,k}}( \phi ) | = \sum_{\nu = 1}^k s_\nu \geq k
    \quad (1 \leq k < n),
\end{equation}
since $S$ is an ordered partition of $n$. This shows \eqref{eq:jnk-compl-bound}
completing the proof.
\end{proof}

\begin{Proposition} \label{prop:jnk-part}
Let $\phi \in \Phi$ and $k, n \geq 1$. Assume that $\phi_n$ has no cycle. Let
\[
  p_\nu = | \Set{ x \in \RD_n : h(x) = \nu } |
    \quad (1 \leq \nu \leq n)
\]
and
\[
  m = \max_{1 \leq \nu \leq n} \, \Set{ \nu : p_\nu \geq 1 }.
\]
Define
\[
  \pi = (p_1, \ldots, p_m).
\]
Then $\pi$ is an ordered partition of $n$ and length $m$ with $1 \leq m \leq n$,
where $p_1, \ldots, p_m \geq 1$. Moreover,
\[
  | J_{n,k}( \phi ) | = n - \sum_{\nu = 1}^k p_\nu
    \quad (1 \leq k \leq m)
\]
and $| J_{n,k}( \phi ) | = 0$ for $k \geq m$.
\end{Proposition}

\begin{proof}
We use and extend the proof of Proposition~\ref{prop:jnk-sets}.
Regarding \eqref{eq:jnk-s-part} the sequence $(s_\nu)_{1 \leq \nu \leq n}$
counts elements of height $1$ up to $n$. Therefore, we observe with $m=r$ that
\[
  \pi = (p_1, \ldots, p_m) = (s_1,\ldots,s_r),
\]
except for the case where $\pi = ( 1, \ldots, 1 )$ and $m=n$.
By \eqref{eq:jnk-compl-union} and \eqref{eq:jnk-sum} it follows that
\[
  | J_{n,k}( \phi ) | = n - \sum_{\nu = 1}^k p_\nu
    \quad (1 \leq k \leq m).
\]
Since $\pi$ is an ordered partition of $n$, we have $| J_{n,m}( \phi ) | = 0$
and consequently that $| J_{n,k}( \phi ) | = 0$ for $k \geq m$.
\end{proof}

\begin{Lemma} \label{lem:part-estim}
Let $\pi = (p_1, \ldots, p_m)$ be an ordered partition of $n$ and length $m$,
where $1 \leq m \leq n$. Then
\[
  \sum_{\nu=1}^m \nu p_\nu \leq n m - \binom{m}{2}.
\]
\end{Lemma}

\begin{proof}
Let $\mathcal{P}_{n,m}$ be the set of ordered partitions of $n$ and length $m$.
Define
\[
  \varrho( \pi ) = \sum_{\nu=1}^m \nu p_\nu \quad (\pi \in \mathcal{P}_{n,m}).
\]
We fix $m \geq 1$ for now and use induction on $n$. For $n = m$, we only have
$\pi = (1, \ldots, 1) \in \mathcal{P}_{n,n}$. Then it follows that
\[
  \varrho( \pi ) = \binom{m+1}{2} = m^2 - \binom{m}{2} \leq n m - \binom{m}{2}.
\]
Now assume the result is true for $n \geq m$.
Let $\pi = (p_1, \ldots, p_m) \in \mathcal{P}_{n,m}$. We set
\begin{equation} \label{eq:part-pi-1}
  \pi'_j = (p_1, \ldots, p_{j-1}, p_j+1, p_{j+1}, \ldots, p_m)
    \quad (1 \leq j \leq m),
\end{equation}
observing that $\pi'_j \in \mathcal{P}_{n+1,m}$. We then obtain that
\begin{equation} \label{eq:part-pi-2}
  \varrho( \pi'_j ) = \varrho( \pi ) + j
    \leq n m + j - \binom{m}{2}
    \leq (n+1) m - \binom{m}{2}
    \quad (1 \leq j \leq m).
\end{equation}
The set $\mathcal{P}_{n+1,m}$ can be constructed from $\mathcal{P}_{n,m}$
using \eqref{eq:part-pi-1}. Since $\pi \in \mathcal{P}_{n,m}$ has been chosen
arbitrarily, we infer that \eqref{eq:part-pi-2} holds for any
$\pi' \in \mathcal{P}_{n+1,m}$ showing the claim for $n+1$.
\end{proof}


\section{Matrix properties}

Recall the $n \times n$ matrix
\[
  E_n(i,j) = \Vect{e}_i \Vect{e}_j^t,
\]
which only has the entry $1$ at row $i$ and column $j$ and zeros elsewhere.

\begin{Lemma} \label{lem:e-matrix}
Let $n \geq 2$ and $a,b,c,d \in \RD_n$. Then
\[
  E_n( a, b ) E_n( c, d ) = \begin{cases}
    E_n( a, d ), & \text{if } b = c, \\
    0,           & \text{else}.
  \end{cases}
\]
\end{Lemma}

\begin{proof}
By definition we can write
\[
  E_n( a, b ) E_n( c, d ) = (\Vect{e}_a \Vect{e}_b^t)(\Vect{e}_c \Vect{e}_d^t)
    = \Vect{e}_a (\Vect{e}_b^t \Vect{e}_c) \Vect{e}_d^t
    = \delta_{bc} \, E_n( a, d ),
\]
using Kronecker's delta.
\end{proof}

\begin{Lemma} \label{lem:matrix-nilpotent}
Let $A$ be an $n \times n$ matrix with $n \geq 2$. If $A$ is nilpotent of
degree $k$, then we have
\[
  \det(I-A) = 1 \quad \text{and} \quad
    (I-A)^{-1} = I + A + \cdots + A^{k-1}.
\]
\end{Lemma}

\begin{proof}
We first consider the decomposition
\[
  I = (I-A)(I + A + \cdots + A^{k-1}),
\]
since $A^k=0$. This shows that $I-A$ is invertible with the inverse as given
above. Since $A$ is nilpotent, there exists a similar matrix $U = T^{-1} A T$,
that is an upper triangular matrix having zeros in its diagonal. Thus, we obtain
that $I - A = T ( I - U ) T^{-1}$ and consequently that $\det(I-A) = 1$.
\end{proof}

\begin{Proposition} \label{prop:matrix-power}
Let $\phi \in \Phi$ and $n \geq 2$. Then we have
\[
  \MA_n( \phi )^k = \sum_{( \phi^k_n(x), x ) \in \RD_n^2} E_n( \phi^k_n(x), x )
    \quad(k \geq 1).
\]
\end{Proposition}

\begin{proof}
We use induction on $k$. For $k = 1$, we infer by \eqref{eq:matrix-mn-def} that
\[
  \MA_n( \phi )
    = \Bigg( \Vect{e}_{\phi_n(1) }, \cdots, \Vect{e}_{\phi_n(n) } \Bigg)
    = \sum_{( \phi_n(x), x ) \in \RD_n^2} E_n( \phi_n(x), x ).
\]
Now assume the result is true for $k$. We then obtain that
\begin{align*}
  \MA_n( \phi )^{k+1} &= \MA_n( \phi ) \MA_n( \phi )^k \\
    &= \sum_{( \phi_n(y), y ) \in \RD_n^2} E_n( \phi_n(y), y )
       \sum_{( \phi^k_n(x), x ) \in \RD_n^2} E_n( \phi^k_n(x), x ) \\
    &= \sum_{\substack{( \phi_n(y), y ) \in \RD_n^2\\
         ( \phi^k_n(x), x ) \in \RD_n^2\\y = \phi^k_n(x)}}
         E_n( \phi_n(y), y ) E_n( \phi^k_n(x), x ) \\
    &= \sum_{( \phi^{k+1}_n(x), x ) \in \RD_n^2} E_n( \phi^{k+1}_n(x), x ).
\end{align*}
In the last two steps we have used Lemma~\ref{lem:e-matrix} to exclude
those products that provide a zero matrix. This shows the claim for $k+1$.
\end{proof}

\begin{Corollary} \label{corl:matrix-estim}
Let $\phi \in \Phi$ and $n \geq 2$. If $\phi_n$ has no cycle, then
\[
  \# \MA_n( \phi )^k = | J_{n,k}( \phi ) | \leq n - k
    \quad (1 \leq k \leq n).
\]
\end{Corollary}

\begin{proof}
By Propositions \ref{prop:jnk-sets} and \ref{prop:matrix-power} we conclude that
\[
  \# \MA_n( \phi )^k
    = \sum_{( \phi^k_n(x), x ) \in \RD_n^2} 1
    = | \Set{ ( x, \phi^k_n(x) ) \in \RD_n^2 } |
    = | J_{n,k}( \phi ) | \leq n - k
    \quad (1 \leq k \leq n). \qedhere
\]
\end{proof}

\begin{Proposition} \label{prop:matrix-disjoint}
Let $\phi \in \Phi$ and $n \geq 2$. If $\phi_n$ has no cycle, then
\[
  \MA_n( \phi )^k \cap \MA_n( \phi )^l = 0 \quad (k \neq l, \, k,l \geq 1).
\]
\end{Proposition}

\begin{proof}
Assume to the contrary that we have
\[
  \MA_n( \phi )^k \cap \MA_n( \phi )^l \neq 0
\]
for some $k \neq l$. We then have an entry of both $\MA_n( \phi )^k$ and
$\MA_n( \phi )^l$, such that
\[
  E_n( \phi^k_n(y), y ) = E_n( \phi^l_n(x), x )
\]
using Proposition~\ref{prop:matrix-power}. This implies that $y = x \in \RD_n$
and consequently $\phi^k_n(x) = \phi^l_n(x) \in \RD_n$. The last condition
gives a contradiction in view of Lemma~\ref{lem:phi-k-neq-l}.
\end{proof}


\section{Proof of theorems}

\begin{proof}[Proof of Theorem~\ref{thm:matrix-cycle}]
The local function $\phi_n$ has a cycle $\Cyc(\phi_n,m,x)$ by assumption. Using
\eqref{eq:matrix-vect} we have
\[
  \MA_n( \phi ) \, \Vect{e}_y = \Vect{e}_{\phi_n(y)}
    \quad (y \in \Cyc(\phi_n,m,x)).
\]
In view of \eqref{eq:cycle-def} and \eqref{eq:cycle-prop}, $\phi_n$ maps
$\Cyc(\phi_n,m,x)$ onto itself in a cyclic way. Define
\[
  \Vect{v} = \sum_{y \in \Cyc(\phi_n,m,x)} \Vect{e}_y.
\]
We then infer that
\[
  \MA_n( \phi ) \, \Vect{v} = \Vect{v}.
\]
As a result, the vector $\Vect{v}$ is an eigenvector of $\MA_n( \phi )$ with
eigenvalue $1$. Consequently, we obtain that
\[
  \det \MI_n( \phi ) = \det( I - \MA_n( \phi ) ) = 0. \qedhere
\]
\end{proof}

\begin{proof}[Proof of Theorem~\ref{thm:matrix-inverse}]
By assumption $\phi_n \in \Phi$ has no cycle.

(1), (2): By Propositions~\ref{prop:matrix-power} and
\ref{prop:matrix-disjoint} the matrices $\MA_n( \phi )^k$ are binary matrices
for $k \geq 1$, which are disjoint for different exponents. The estimate
\begin{equation} \label{eq:matrix-estim}
  \# \MA_n( \phi )^k \leq n-k \quad (1 \leq k \leq n)
\end{equation}
is given by Corollary~\ref{corl:matrix-estim}. This implies that $\MA_n( \phi)^n$
is the zero matrix and consequently $\MA_n( \phi )$ is nilpotent of degree
at most $n$.

(3), (4): By definition we have $\MI_n( \phi ) = I - \MA_n( \phi )$.
Since $\MA_n( \phi )$ is nilpotent of degree at most $n$, we obtain by
Lemma~\ref{lem:matrix-nilpotent} that $\det \MI_n( \phi ) = 1$ and
\begin{equation} \label{eq:matrix-inverse}
  \MI_n( \phi )^{-1} = I + \MA_n( \phi ) + \cdots + \MA_n( \phi )^{n-1}.
\end{equation}
Recall that $\MA_n( \phi )^k \, \cap \, I = 0$, because $\phi_n$ has no fixed
points in $\RD_n$, and that powers of $\MA_n( \phi )$ are disjoint for different
exponents. Therefore $\MI_n( \phi )^{-1}$ is a binary matrix composed of binary
matrices given on the right-hand side of \eqref{eq:matrix-inverse}. Counting
entries of these matrices above, we conclude by using the estimate in
\eqref{eq:matrix-estim} that
\[
  \# \MI_n( \phi )^{-1} = n + \sum_{k=1}^{n-1} \# \MA_n( \phi )^k
    \leq \binom{n+1}{2}. \qedhere
\]
\end{proof}

\begin{proof}[Proof of Theorem~\ref{thm:matrix-part}]
By assumption $\phi_n \in \Phi$ has no cycle. The properties of $\pi$ are given
by Proposition~\ref{prop:jnk-part}.

(1), (2): By Proposition~\ref{prop:matrix-power} we have
\begin{equation} \label{eq:matrix-pow-en}
  \MA_n( \phi )^k = \sum_{( \phi^k_n(x), x ) \in \RD_n^2} E_n( \phi^k_n(x), x )
    \quad(k \geq 1).
\end{equation}
From Proposition~\ref{prop:jnk-part} and Corollary~\ref{corl:matrix-estim}
we infer that
\begin{equation} \label{eq:matrix-pow-part}
  \# \MA_n( \phi )^k = | J_{n,k}( \phi ) |
    = n - \sum_{\nu = 1}^k p_\nu
    \quad (1 \leq k \leq m).
\end{equation}
Since $| J_{n,m}( \phi ) | = 0$, the matrix $\MA_n( \phi )$ is nilpotent of
degree $m$.

(3): Since the nilpotent degree of $\MA_n( \phi )$ is $m \leq n$, we have by
Lemma~\ref{lem:matrix-nilpotent} that
\begin{equation} \label{eq:matrix-inverse-2}
  \MI_n( \phi )^{-1} = I + \MA_n( \phi ) + \cdots + \MA_n( \phi )^{m-1}.
\end{equation}
As already argued in \eqref{eq:matrix-inverse} and below, the matrices of the
right-hand side of \eqref{eq:matrix-inverse-2} are disjoint to each other.
Counting entries we obtain by means of \eqref{eq:matrix-pow-part} that
\begin{equation} \label{eq:matrix-inv-count-1}
  \# \MI_n( \phi )^{-1}
    = n + \sum_{k=1}^{m-1} \# \MA_n( \phi )^k
    = \sum_{k=0}^{m-1} \sum_{\nu = k+1}^m p_\nu
    = \sum_{\nu = 1}^m \nu p_\nu,
\end{equation}
using the fact that $\pi$ is a partition of $n$ and thus
\[
  n - \sum_{\nu = 1}^k p_\nu = \sum_{\nu = k+1}^m p_\nu
    \quad (0 \leq k < m).
\]
Alternatively, from \eqref{eq:matrix-inverse-2} and using \eqref{eq:matrix-pow-en},
we derive that
\begin{align}
  \MI_n( \phi )^{-1}
    &= I + \sum_{k=1}^{m-1}
      \sum_{( \phi^k_n(x), x ) \in \RD_n^2} E_n( \phi^k_n(x), x ) \nonumber \\
    &= I + \sum_{x \in \RD_n} \sum_{\nu = 1}^{h(x)-1} E_n( \phi^\nu_n(x), x ).
       \label{eq:matrix-inverse-3}
\end{align}
This implies that
\begin{equation} \label{eq:matrix-inv-count-2}
   \# \MI_n( \phi )^{-1} = \sum_{x \in \RD_n} h(x).
\end{equation}
Finally, combining \eqref{eq:matrix-inv-count-1} and \eqref{eq:matrix-inv-count-2},
we achieve by applying Lemma~\ref{lem:part-estim} that
\[
  \# \MI_n( \phi )^{-1}
    = \sum_{x \in \RD_n} h(x)
    = \sum_{\nu=1}^m \nu p_\nu \leq n m - \binom{m}{2}.
\]
It remains to show the structure of the inverse matrix $\MI_n( \phi )^{-1}$. Let
\[
  \MI_n( \phi )^{-1} = \Bigg( \Vect{v}_1, \ldots, \Vect{v}_n \Bigg).
\]
From \eqref{eq:matrix-inverse-3} we conclude for $j=x$ that
\[
  \Vect{v}_j = \Vect{e}_j + \sum_{\nu = 1}^{h(j)-1} \Vect{e}_{\phi_n^\nu(j)}
    \quad (1 \leq j \leq n). \qedhere
\]
\end{proof}


\section{Simple patterns}
\label{sec:simple-patterns}

In this section we shall give some applications. Define an $n \times n$
sub-diagonal matrix as
\[
  D_{n,k} = \sum_{j=1}^{n-k} E_n( j+k,j )
    \quad (0 \leq k < n)
\]
and as a zero matrix otherwise.

\begin{Proposition} \label{prop:phi-x+1}
Let $\phi$ be induced by $f(x) = x + 1$. Then $\phi \in \Phi$ and has no cycle.
We have the following statements for $n \geq 2$:
\begin{enumerate}
\item The matrix $\MA_n( \phi )$ is nilpotent of degree $n$.
\item We have
\[
  \MA_n( \phi )^k = D_{n,k}, \quad
    \# \MA_n( \phi )^k = n-k
    \quad (1 \leq k \leq n).
\]
\item The inverse $\MI_n( \phi )^{-1}$ is a full lower triangular matrix whose
entire entries are $1$. Consequently, $\# \MI_n( \phi )^{-1} = \binom{n+1}{2}$.
\end{enumerate}
\end{Proposition}

\begin{proof}
It is easily seen that $\phi^k(x) = x + k$ for $x \in \NN$. Therefore, $\phi$
cannot have a cycle or a fixed point in $\NN$, hence $\phi \in \Phi$. By
Theorem~\ref{thm:matrix-part} we then obtain that
\[
  \MA_n( \phi )^k
    = \sum_{( \phi^k_n(x), x ) \in \RD_n^2} E_n( \phi^k_n(x), x )
    = D_{n,k}
    \quad (1 \leq k \leq n),
\]
implying that $\# \MA_n( \phi )^k = n-k$. This shows that $\MA_n( \phi )$
is nilpotent of degree $n$. Theorem~\ref{thm:matrix-inverse} provides that
\[
  \MI_n( \phi )^{-1}
    = I + \sum_{k=1}^{n-1} \MA_n( \phi )^k
    = \sum_{k=0}^{n-1} D_{n,k},
\]
which is a full lower triangular matrix having only entries with $1$.
As a consequence, we have $\# \MI_n( \phi )^{-1} = \binom{n+1}{2}$.
\end{proof}

\begin{Remark}
If one chooses $f(x) = x + t$ with $t \in \ZZ \backslash \Set{0}$, where $f$ is
defined on the subdomain $\NN_{> |t|}$ in case $t < 0$, then one similarly
obtains the following shapes of the matrix $\MI_n( \phi )^{-1}$, where $\phi$
is induced by $f$:
\begin{enumerate}
\item Case $t=1$: $\MI_n( \phi )^{-1}$ is a full lower triangular matrix shown
by Proposition~\ref{prop:phi-x+1}.
\item Case $t>1$: $\MI_n( \phi )^{-1}$ is a lower triangular matrix with
\[
  \MI_n( \phi )^{-1} = \sum_{j=0}^{\lfloor \frac{n-1}{t} \rfloor} D_{n,jt}.
\]
For $t=2$ this gives a checkerboard pattern.
\item Case $t=-1$: $\MI_n( \phi )^{-1}$ is a full upper triangular matrix.
\item Case $t<-1$: $\MI_n( \phi )^{-1}$ is an upper triangular matrix,
      which equals the transposed matrix of the case $|t| > 1$.
\end{enumerate}
The cases $t \neq 1$ will be left to the reader. See
Figure~\ref{fig:matrix-x-5} for an example.
\end{Remark}

\begin{minipage}{0.95\textwidth}
\captionsetup{type=figure}
\captionof{figure}{Matrices $\MI_n( \phi )$ and $\MI_n( \phi )^{-1}$}
\label{fig:matrix-x-5}
\begin{center}
  \includegraphics[width=5.4cm]{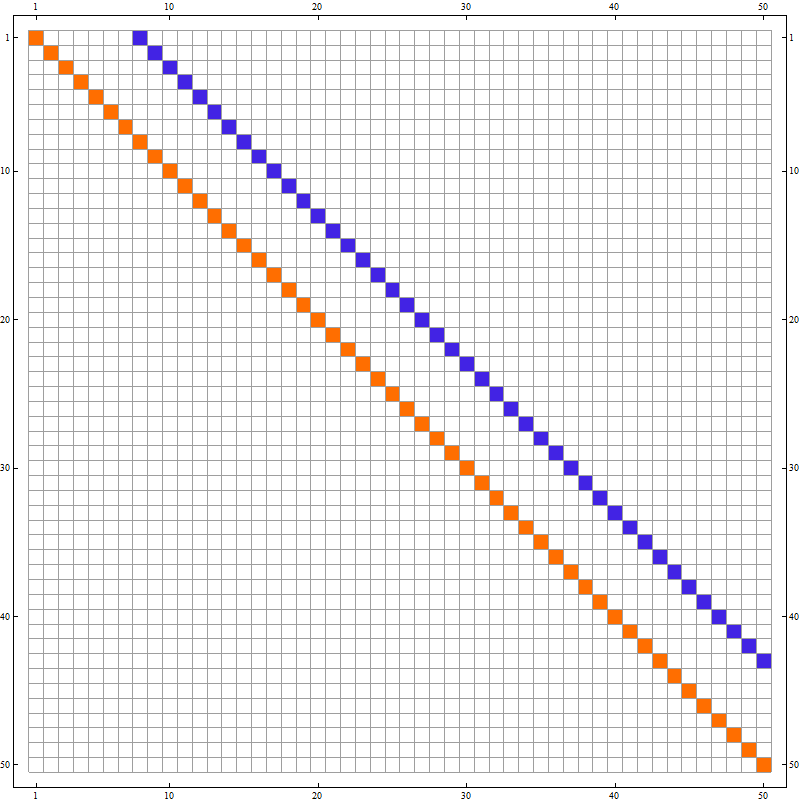}
  \includegraphics[width=5.4cm]{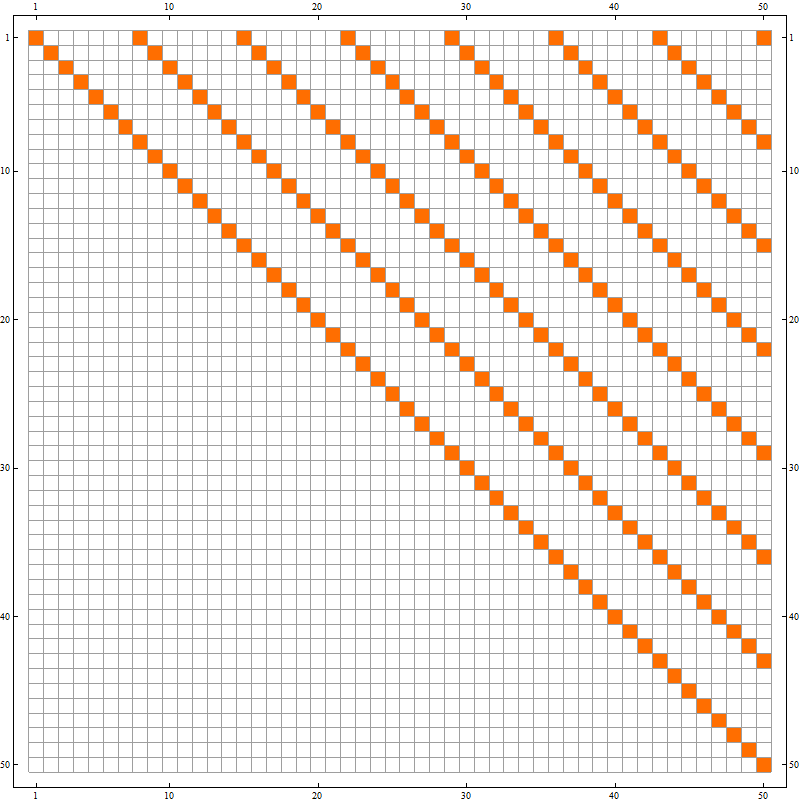}
\end{center}
\begin{center} \small
  Case $n=50$: $\phi$ is induced by $f(x)=x-7$.\\
  Red entries: $1$, blue entries: $-1$.
\end{center}
\end{minipage}
\bigskip

For a second example, let $\PP = \Set{q_1,q_2,\ldots}$ be the set of primes and
$\Pi(x)$  be the prime-counting function. Define
\[
  \omega: \PP \to \PP, \quad \omega( q_j ) = q_{j+1}
\]
giving the next prime.

\begin{Proposition}
Let $\phi$ be induced by $\omega$. Then $\phi \in \Phi$ and has no cycle.
We have the following statements for $n \geq 2$:
\begin{enumerate}
\item The matrix $\MA_n( \phi )$ is nilpotent of degree $\Pi(n)$.
\item We have
\[
  \MA_n( \phi )^k
    = \sum_{j=1}^{\Pi(n)-k} E_n( q_{j+k}, q_j ),
    \quad \# \MA_n( \phi )^k = \Pi(n)-k,
    \quad (1 \leq k \leq \Pi(n)).
\]
\item We have
\[
  \MI_n( \phi )^{-1} = I + \sum_{1 \leq j < i \leq \Pi(n) } E_n( q_i, q_j ),
    \quad \# \MI_n( \phi )^{-1} = n + \binom{\Pi(n)}{2}.
\]
\end{enumerate}
\end{Proposition}

\begin{proof}
Since $\omega^k(q_j)= q_{j+k} > q_j$ for $q_j \in \PP$ and $j, k \geq 1$, the
induced function $\phi$ cannot have a cycle or a fixed point in $\NN$. Thus, we
have $\phi \in \Phi$. By construction, $\phi(x) = 0$ for $x \in \NN_0
\backslash \PP$. Let $N = \Pi(n)$, then $\RD_n \cap \PP = \Set{q_1, \ldots, q_N}$.
Using these properties and Theorem~\ref{thm:matrix-part} we infer that
\[
  \MA_n( \phi )^k
    = \sum_{( \phi^k_n(x), x ) \in \RD_n^2} E_n( \phi^k_n(x), x )
    = \sum_{j=1}^{N-k} E_n( q_{j+k}, q_j )
    \quad (1 \leq k \leq N).
\]
As a result, $\# \MA_n( \phi )^k = N-k$ for $1 \leq k \leq N$, implying that
$\MA_n( \phi )$ is nilpotent of degree $N$. The last part follows by
Theorem~\ref{thm:matrix-inverse} and reordering the above sums that
\[
  \MI_n( \phi )^{-1} = I + \sum_{k=1}^{N-1} \MA_n( \phi )^k
    = I + \sum_{k=1}^{N-1} \sum_{j=1}^{N-k} E_n( q_{j+k}, q_j )
    = I + \sum_{1 \leq j < i \leq N } E_n( q_i, q_j ).
\]
Counting entries of $\MI_n( \phi )^{-1}$ in the equation above,
we finally obtain that
\[
  \# \MI_n( \phi )^{-1} = n + \sum_{k=1}^{N-1} ( N-k ) = n + \binom{N}{2}.
    \qedhere
\]
\end{proof}

\begin{minipage}{0.95\textwidth}
\captionsetup{type=figure}
\captionof{figure}{Matrices $\MI_n( \phi )$ and $\MI_n( \phi )^{-1}$}
\label{fig:matrix-primes}
\begin{center}
  \includegraphics[width=5.4cm]{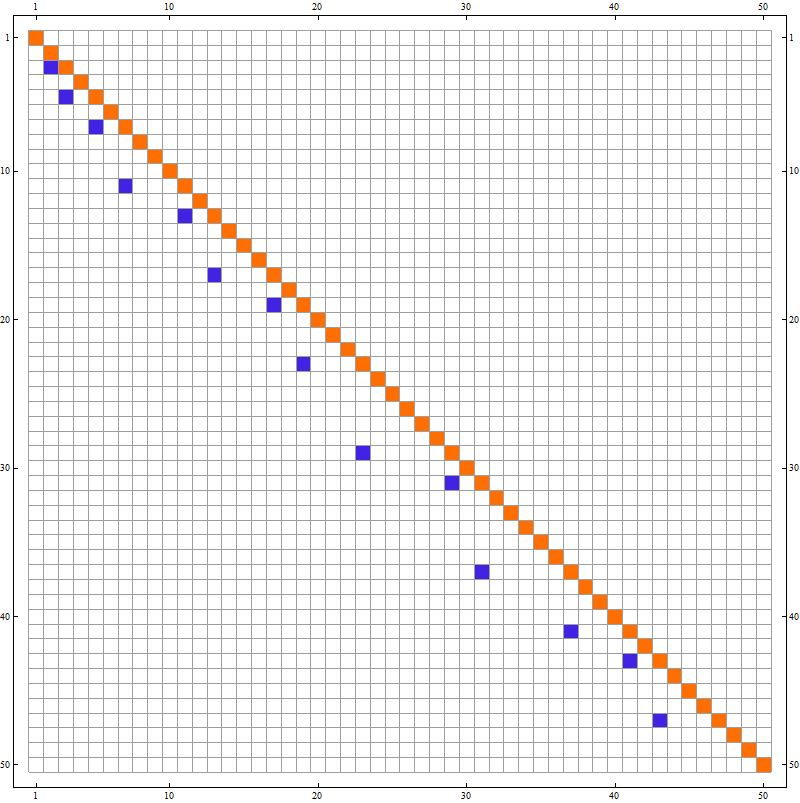}
  \includegraphics[width=5.4cm]{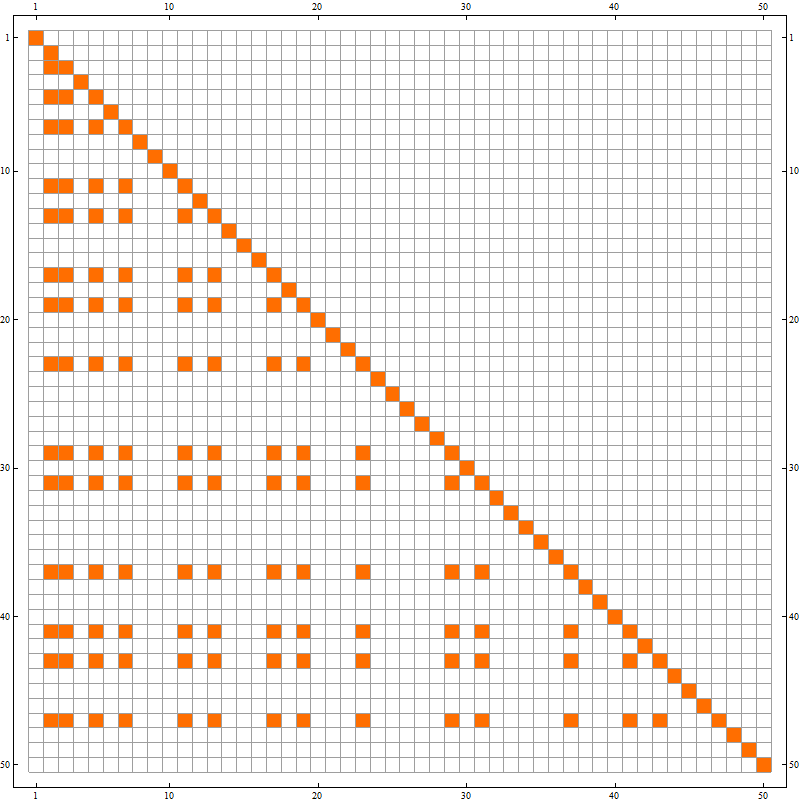}
\end{center}
\begin{center} \small
  Case $n=50$: $\phi$ is induced by $\omega$.\\
  Red entries: $1$, blue entries: $-1$.
\end{center}
\end{minipage}


\section{\texorpdfstring{The $3x+1$ problem}{The 3x+1 problem}}
\label{sec:3x+1-problem}

A variant of the Collatz function may be defined by
\[
  c(x) = \begin{cases}
    x/2,      & \text{if $x$ is even}, \\
    (3x+1)/2, & \text{if $x$ is odd}.
  \end{cases}
\]
The behavior of the iterations of this function is known as the $3x+1$ problem.
It is still an open problem to decide, whether a sequence of iterations
$\mathcal{I}_c(x) = (c^k(x))_{k \geq 1}$, starting from a positive integer $x$,
eventually returns to $1$ entering a trivial cycle \Set{1,2} afterwards. It is
conjectured that all such iterations eventually return to $1$. For a wide
survey of the $3x+1$ problem see Lagarias \cite{Lagarias:2010}.
\smallskip

There are three possible cases of the behavior of a sequence $\mathcal{I}_c(x)$:
\begin{enumerate}
\item It eventually enters the trivial cycle \Set{1,2}.
\item It eventually enters a cycle other than \Set{1,2}.
\item It is unbounded.
\end{enumerate}
\smallskip

We can establish a connection between the cases (1), (2), and the given theory
in the former sections. To get rid of the trivial cycle \Set{1,2}, we define
\[
  \phi_c(x) = \begin{cases}
    0,         & \text{if $x \leq 2$}, \\
    x/2,       & \text{if $x > 2$ is even}, \\
    (3x+1)/2,  & \text{if $x > 2$ is odd}. \\
  \end{cases}
\]
Then we have $\phi_c \in \Phi$. As a result of Corollary~\ref{corl:matrix-cycle},
if $\phi_c$ has a cycle, then there exists an integer $N \geq 2$ such that
\[
  \det \MI_n( \phi_c ) = 0 \quad (n \geq N).
\]

See Figure~\ref{fig:matrix-3x+1} for the simple shape of $\MI_n( \phi_c )$
and the complicated shape of its inverse in the case $n=50$.
Regarding the local function $\phi_{c,n}$ for this case, we can compute the
following parameters using Theorem~\ref{thm:matrix-part}:
\[
  \pi = ( 10, 4, 3, 3, 3, 3, 4, 2, 2, 2, 2, 2, 3, 2, 2, 1, 1, 1 ),
    \quad m = 18.
\]
Consequently, the nilpotent degree of $\MA_n( \phi_c )$ is $18$ and
\[
  \# \MI_n( \phi_c )^{-1} = \sum_{\nu=1}^m \nu p_\nu = 348.
\]

\begin{Remark} \label{rem:3x+1}
Zeilberger \cite{Zeilberger:2014} asked for an evaluation of determinants of
certain $2d \times 2d$ matrices $M(d)$ occurring in an enumeration problem.
Actually, this was intended as a semi-joke \cite{Zeilberger:2014b}, because
these matrices were disguised intentionally, hiding their close relationship to
the $3x+1$ problem at first glance.

Chapman~\cite{Chapman:2014} pointed out, that one obtains, after swapping
columns of $M(d)$, that
\[
  \det M(d) = (-1)^d \det(I - N(d)),
\]
where $N(d)$ is the matrix describing the local iteration of
\[
  \psi(x) = \begin{cases}
    (x-1)/2, & \text{if $x$ is odd}, \\
    3x/2+1,  & \text{if $x$ is even}.
  \end{cases}
\]
Chapman showed that either $N(d)$ is nilpotent and $\det M(d) = (-1)^d$ or
$N(d)$ has an eigenvector with eigenvalue 1 and $\det M(d) = 0$, induced by a
cycle of $\psi$. Furthermore, he established the connection to the $3x+1$
problem via $c(x) = \psi(x-1)+1$.
\end{Remark}

\begin{minipage}{0.95\textwidth}
\captionsetup{type=figure}
\captionof{figure}{Matrices $\MI_n( \phi_c )$ and $\MI_n( \phi_c )^{-1}$}
\label{fig:matrix-3x+1}
\begin{center}
  \includegraphics[width=5.4cm]{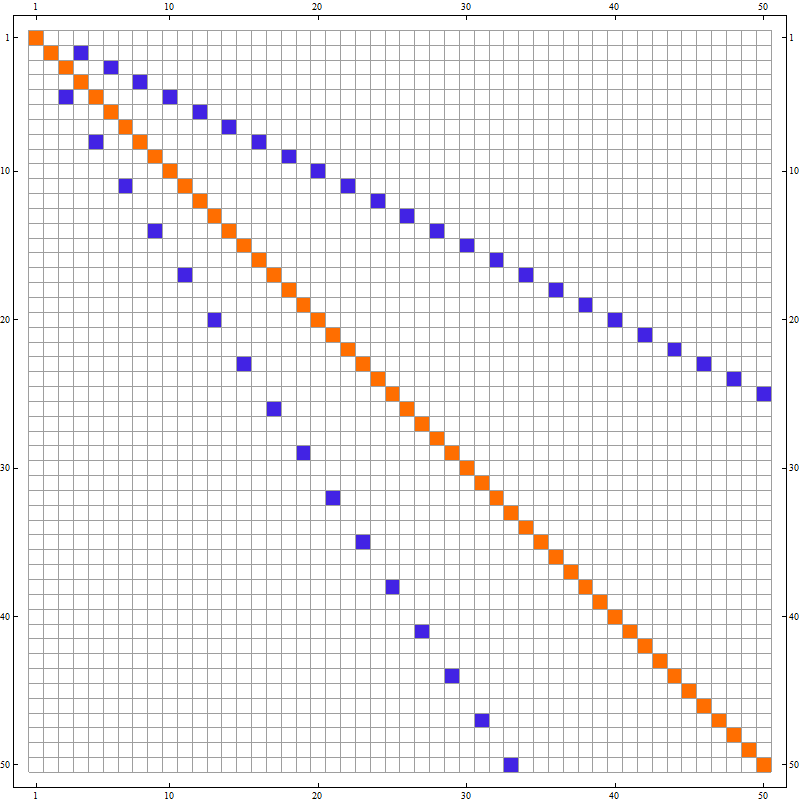}
  \includegraphics[width=5.4cm]{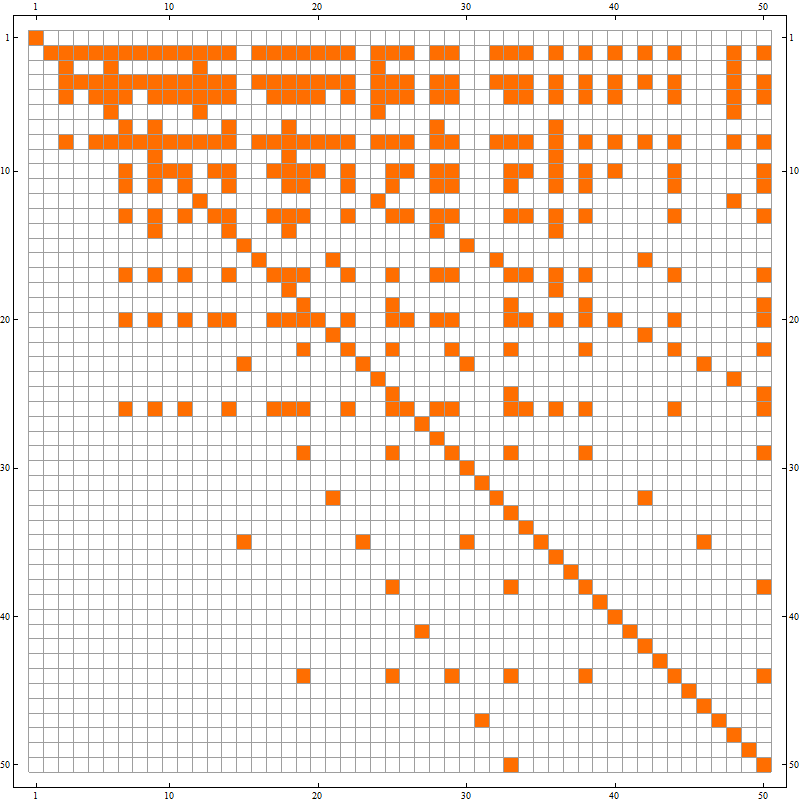}
\end{center}
\begin{center} \small
  Case $n=50$. Red entries: $1$, blue entries: $-1$.
\end{center}
\end{minipage}
\bigskip

As a last example, we consider a more complicated function with $r=3$ branches:
\[
  \phi_r(x) = \begin{cases}
    0,         & \text{if $x \leq 1$}, \\
    x/3,       & \text{if $x > 1$ and $x \equiv 0 \!\!\! \pmod{3}$}, \\
    (2x+1)/3,  & \text{if $x > 1$ and $x \equiv 1 \!\!\! \pmod{3}$}, \\
    (5x-1)/3,  & \text{if $x > 1$ and $x \equiv 2 \!\!\! \pmod{3}$}. \\
  \end{cases}
\]
Such functions are called \textsl{generalized Collatz} functions or
\textsl{residue-class-wise affine} functions, which can be defined for any
$r \geq 2$, $r$ being the number of branches, respectively, residue classes
(cf.~\cite[(4.1), p.~12]{Lagarias:2010}).

The modification here, that $\phi_r(1) = 0$, is only to prevent a fixed point
at $x=1$. In this way, we have $\phi_r \in \Phi$. Again, we compute the
parameters of the local function $\phi_{r,n}$ for $n=50$:
\[
  \pi = ( 8, 2, 5, 5, 6, 4, 5, 8, 4, 2, 1 ), \quad m = 11.
\]
Consequently, the nilpotent degree of $\MA_n( \phi_r )$ is $11$ and
\[
  \# \MI_n( \phi_r )^{-1} = \sum_{\nu=1}^m \nu p_\nu = 267.
\]

\begin{minipage}{0.95\textwidth}
\captionsetup{type=figure}
\captionof{figure}{Matrices $\MI_n( \phi_r )$ and $\MI_n( \phi_r )^{-1}$}
\label{fig:matrix-phi-3}
\begin{center}
  \includegraphics[width=5.4cm]{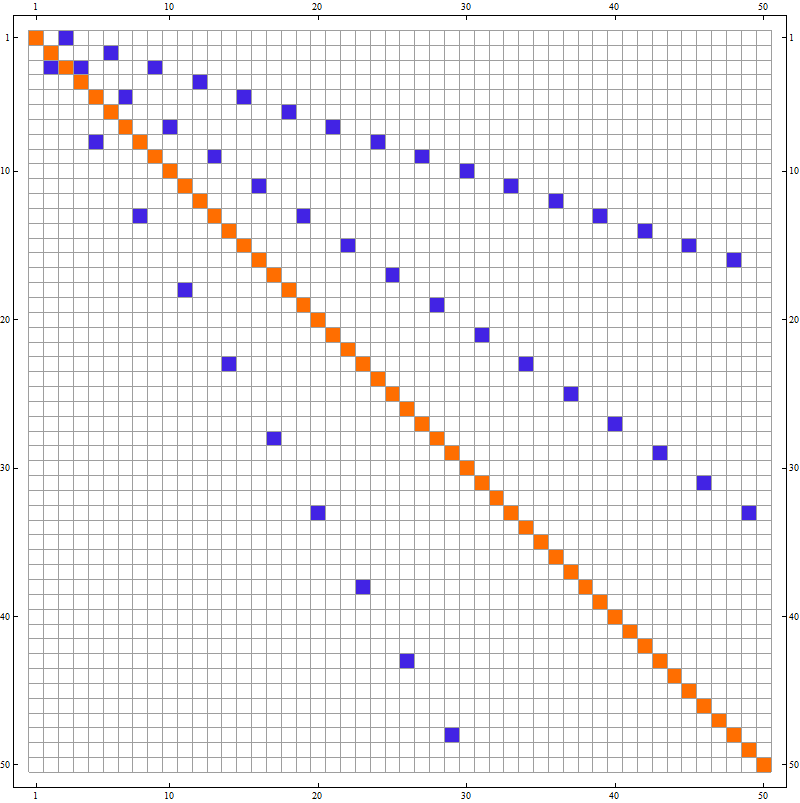}
  \includegraphics[width=5.4cm]{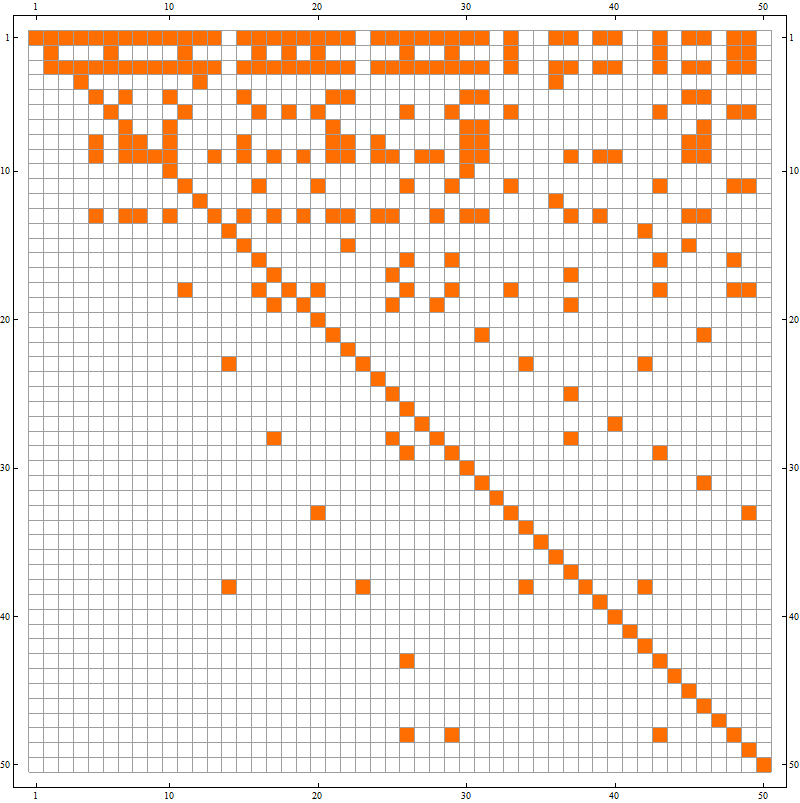}
\end{center}
\begin{center} \small
  Case $n=50$. Red entries: $1$, blue entries: $-1$.
\end{center}
\end{minipage}
\bigskip

The $3x+1$ problem, as treated in Remark~\ref{rem:3x+1}, was the starting point
for the author to give a general theory here. All computations were performed
using \textsl{Mathematica}.

\section*{Acknowledgments}

We would like to thank Pieter Moree for helpful comments.


\bibliographystyle{amsplain}

\bigskip

\end{document}